\documentclass[review]{elsarticle}
\usepackage{amsmath, amssymb, amsthm, color}
\usepackage{lineno,hyperref, graphicx}
\usepackage[english]{babel}
\modulolinenumbers[5]
\def\R{\mathbb R}  \def\N{\mathbb N} \def\C{\mathbb C}
\newtheorem{thm}{Theorem}[section]
\newtheorem{lem}[thm]{Lemma}
\newtheorem{cor}[thm]{Corollary}
\newtheorem{defn}[thm]{Definition}

\newtheorem{rem}[thm]{Remark}
\numberwithin{equation}{section}
\begin{document}
%\maketitle

\begin{frontmatter}
\title{The Cauchy problem for the critical inhomogeneous nonlinear Schr\"{o}dinger equation in $H^{s}(\R^{n})$}
\author{JinMyong An}
\author[]{JinMyong Kim\corref{mycorrespondingauthor}}
\ead{jm.kim0211@ryongnamsan.edu.kp}
\cortext[mycorrespondingauthor]{Corresponding author}
\address{Faculty of Mathematics, {\bf Kim Il Sung} University, Pyongyang, Democratic People's Republic of Korea}
\begin{abstract}

In this paper, we study the Cauchy problem for the critical inhomogeneous nonlinear Schr\"{o}dinger (INLS) equation
\[iu_{t} +\Delta u=|x|^{-b} f(u), ~u(0)=u_{0} \in H^{s} (\R^{n} ),\]
where $n\ge3$, $1\le s<\frac{n}{2} $, $0<b<2$ and $f(u)$ is a nonlinear function that behaves like $\lambda \left|u\right|^{\sigma } u$ with $\lambda \in \C$ and $\sigma =\frac{4-2b}{n-2s} $. We establish the local well-posedness as well as the small data global well-posedness and scattering in $H^{s} (\R^{n} )$ with $1\le s<\frac{n}{2}$ for the critical INLS equation under some assumption on $b$. To this end, we first establish various nonlinear estimates by using fractional Hardy inequality and then use the contraction mapping principle based on Strichartz estimates.
\end{abstract}
\begin{keyword}
Inhomogeneous nonlinear Schr\"{o}dinger equation\sep Critical\sep Well-posedness \sep Scattering\sep Strichartz estimates \sep Hardy inequality\\
2020 MSC: 35Q55\sep 46E35
\end{keyword}
\end{frontmatter}

%%%%%%%%%%%%%%%%%%%%%%%%%%%%%%%%%%%%%%%%%%%%%%%%%%%%%%%%%%%%%%%%%%%%%%%%%%%%%

\section{Introduction}

In the present paper, we consider the Cauchy problem for the inhomogeneous nonlinear Schr\"{o}dinger (INLS) equation
\begin{equation} \label{GrindEQ__1_1_}
\left\{\begin{array}{l} {iu_{t} +\Delta u=|x|^{-b} f(u),} \\ {u(0,\; x)=u_{0} (x),} \end{array}\right.
\end{equation}
where $u:\;\R\times\R^{n}\to \C$, $u_{0}:\;\R^{n}\to \C$, $b>0$ and $f$ is of class ${\rm X} \left(\sigma ,s,b\right)$.
The class ${\rm X} \left(\sigma ,s,b\right)$ is defined as follows.
\begin{defn}[\cite{AK21}]\label{defn_1.1.}
\textnormal{Let $f:\C\to \C$, $s\ge 0$, $0\le b<2$ and $\left\lceil s\right\rceil $ denote the minimal integer which is larger than or equals to $s$. For $k\in \N$, let $k$-th order complex partial derivative of $f$ be defined under the identification $\C=\R^{2} $ (see Section 2).
We say that $f$ is of class ${\rm X} \left(\sigma ,s,b\right)$ if it satisfies one of the following conditions:
\begin{itemize}
  \item $f\left(z\right)$ is a polynomial in $z$ and $\bar{z}$ satisfying $1<\deg \left(f\right)=1+\sigma \le 1+\frac{4-2b}{n-2s} $ $(1<\deg \left(f\right)=1+\sigma <\infty $, if $s\ge \frac{n}{2} )$
  \item $f\in C^{\max \left\{\left\lceil s\right\rceil ,1\; \right\}} \left(\C\to \C\right)$ and
    \begin{equation} \label{GrindEQ__1_2_}
      \left|f^{\left(k\right)} (z)\right|\lesssim\left|z\right|^{\sigma +1-k}
    \end{equation}
   for any $0\le k\le \max \left\{\left\lceil s\right\rceil ,1\; \right\}$ and $z\in \C$, where $\left\lceil s\right\rceil \le \sigma +1\le 1+\frac{4-2b}{n-2s} $ ($\left\lceil s\right\rceil \le \sigma +1<\infty $, if $s\ge \frac{n}{2}$).
\end{itemize}}
\end{defn}

\begin{rem}[\cite{AK21}]\label{rem 1.2.}
\textnormal{Let $s\ge 0$ and $0<b<2$. Assume that $0<\sigma \le \frac{4-2b}{n-2s} $, if $s<\frac{n}{2} $, and that $0<\sigma <\infty $, if $s\ge \frac{n}{2} $. If $\sigma $ is not an even integer, assume further $\left\lceil s\right\rceil \le \sigma +1$. Then we can easily see that $f(u)=\lambda \left|u\right|^{\sigma } u$ with $\lambda \in \C$ is a model case of class ${\rm X} \left(\sigma ,s,b\right)$.}
\end{rem}

The INLS equation \eqref{GrindEQ__1_1_} arises in nonlinear optics for the propagation of laser beam and it has been widely studied by many authors. For the physical background and applications of \eqref{GrindEQ__1_1_}, we refer the reader to \cite{BKMT11, BKVLMT12, G00, LT94}.

The INLS equation \eqref{GrindEQ__1_1_} has the following equivalent form:
\begin{equation} \label{GrindEQ__1_3_}
u(t)=S(t)u_{0} -i\int _{0}^{t}S(t-\tau )|x|^{-b} f\left(u(\tau )\right)d\tau  ,
\end{equation}
where $S(t)=e^{it\Delta } $ is the Schr\"{o}dinger semi-group. It is convenient to introduce the following notation which is used throughout the paper:
\begin{equation} \label{GrindEQ__1_4_}
\sigma _{s} =\left\{\begin{array}{l} {\frac{4-2b}{n-2s} ,\; 0\le s<\frac{n}{2}, } \\ {\infty ,\; s\ge \frac{n}{2} }. \end{array}\right.
\end{equation}

When $0\le s<\frac{n}{2} $, $\sigma _{s} $ is said to be a $H^{s}$-critical power. If $s\ge 0$, $\sigma <\sigma _{s} $ is said to be a $H^{s} $-subcritical power. See \cite{AK21} for example. We say that a pair $(\gamma(p),\;p)$ is admissible, if
\begin{equation} \label{GrindEQ__1_5_}
\left\{\begin{array}{l} {2\le p\le \frac{2n}{n-2} ,\;n\ge 3,} \\ {2\le p<\infty , \;n=2,} \\ {2\le p\le \infty , \;n=1,} \end{array}\right.
\end{equation}
and
\begin{equation} \label{GrindEQ__1_6_}
\frac{2}{\gamma (p)} =\frac{n}{2} -\frac{n}{p} .
\end{equation}

The local and global well-posedness as well as the scattering and blow-up in the energy space $H^{1} (\R^{n} )$ for \eqref{GrindEQ__1_1_} with $f(u)=\lambda \left|u\right|^{\sigma } u$ have been widely studied by many authors. See, for example, \cite{AC21, C21, C03, CG16, D18, D19, F16, FG17, FG19, GS08, G12, GM21} and the references therein.

Meanwhile, the local and global well-posedness in the Sobolev space $H^{s} (\R^{n} )$ for \eqref{GrindEQ__1_1_} have also been investigated.
Guzm\'{a}n \cite{G17} established the local and global well-posedness in $H^{s} (\R^{n} )$ with $0\le s\le \min \left\{1,\;\frac{n}{2} \right\}$ for \eqref{GrindEQ__1_1_} with $f(u)=\lambda \left|u\right|^{\sigma } u$. More precisely, he proved that:
\begin{itemize}
  \item if $0<\sigma <\sigma _{0} $, and $0<b<\min \{ 2,\;n\} $, then \eqref{GrindEQ__1_1_} is globally well-posed in $L^{2} (\R^{n} )$;
  \item  if $0<s\le \min \left\{1,\;\frac{n}{2} \right\}$, $0<b<\tilde{2}$ and $0<\sigma <\sigma _{s} $, then \eqref{GrindEQ__1_1_} is locally well-posed in $H^{s}(\R^{n})$;
  \item if $0<s\le \min \left\{1,\;\frac{n}{2} \right\}$, $0<b<\tilde{2}$ and $\sigma _{0} <\sigma <\sigma _{s} $, then \eqref{GrindEQ__1_1_} is globally well-posed in $H^{s}(\R^{n})$ for small initial data, where
\end{itemize}
\begin{equation} \label{GrindEQ__1_7_}
\tilde{2}=\left\{\begin{array}{l} {\frac{n}{3} ,\;n=1,\;2,\;3,} \\ {2,\;n\ge 4.} \end{array}\right.
\end{equation}
Later, An-Kim \cite{AK21} improved the local well-posedness result of \cite{G17} by proving that the INLS equation \eqref{GrindEQ__1_1_} is locally well-posed in $H^{s}(\R^{n})$ if $0\le s<\min \left\{\frac{n}{2} +1,\; n\right\}$, $0<b<\min \left\{2,\; n-s,\; 1+\frac{n-2s}{2} \right\}$, $0<\sigma <\sigma _{s} $ and $f$ is of class ${\rm X} \left(\sigma ,s,b\right)$.
But the authors in \cite{G17, AK21} only dealt with the $H^{s}$-subcritical case and the local well-posedness for the INLS equation \eqref{GrindEQ__1_1_} in the $H^{s}$-critical case, i.e. $\sigma =\frac{4-2b}{n-2s} $ with $0\le s<\frac{n}{2} $ was not known until very recently. See Remark 1.7 of \cite{G17} and Remark 1.5 of \cite{AK21} for example.

The purpose of this paper is to establish the local well-posedness as well as the small data global well-posedness and scattering in $H^{s}(\R^{n})$ with $1\le s<\frac{n}{2}$ for the INLS equation \eqref{GrindEQ__1_1_} in the $H^{s}$-critical case. To arrive at this goal, we first establish various nonlinear estimates by using fractional Hardy inequality and then use the contraction mapping principle based on Strichartz estimates.

The main results of this paper are the following two theorems.

\begin{thm}\label{thm 1.3.}
Let $n\ge 3$, $1\le s<\frac{n}{2} $, $0<b<{\rm 1}+\frac{n-2s}{2} $ and $\sigma =\frac{4-2b}{n-2s} $. Assume that $f$ is of class ${\rm X} \left(\sigma ,s,b\right)$. Assume further that one of the following conditions is satisfied:
\begin{itemize}
  \item $s\in \N$ and $b<\frac{4s}{n} $,
  \item $s\notin \N$, $n\ge 4$ and $b<\frac{6s}{n} -1$.
\end{itemize}
Then for any $u_{0} \in H^{s} (\R^{n})$, there exists $T=T\left(u_{0} \right)>0$ such that \eqref{GrindEQ__1_1_} has a unique solution
\begin{equation} \label{GrindEQ__1_8_}
u\in L^{\gamma \left(r\right)} \left(\left[-T,\;T\right],\;H_{r}^{s} (\R^{n})\right),
\end{equation}
where $\left(\gamma \left(r\right),\;r\right)$ is an admissible pair satisfying
\begin{equation} \label{GrindEQ__1_9_}
r=\frac{2n\sigma +2n}{n\sigma +n-2} .
\end{equation}
Moreover, for any admissible pair $\left(\gamma \left(p\right),\;p\right)$, we have
\begin{equation} \label{GrindEQ__1_10_}
u\in L^{\gamma \left(p\right)} \left(\left[-T,\;T\right],\;H_{p}^{s} (\R^{n})\right).
\end{equation}
If $\left\| u_{0} \right\| _{\dot{H}^{s}(\R^{n}) } $ is sufficiently small, then the above solution is a global one and
\begin{equation} \label{GrindEQ__1_11_}
\left\| u\right\| _{L^{\gamma \left(p\right)} \left(\R,\;\dot{H}_{p}^{s} (\R^{n})\right)} \lesssim\left\| u\right\| _{\dot{H}^{s}(\R^{n}) } ,
\end{equation}
\begin{equation} \label{GrindEQ__1_12_}
\left\| u\right\| _{L^{\gamma \left(p\right)} \left(\R,\;H_{p}^{s} (\R^{n})\right)} \lesssim\left\| u\right\| _{H^{s} (\R^{n})} ,
\end{equation}
for any admissible pair $\left(\gamma \left(p\right),\;p\right)$. Furthermore, there exist $u_{0}^{\pm } \in H^{s} (\R^{n} )$ such that
\begin{equation} \label{GrindEQ__1_13_}
{\mathop{\lim }\limits_{t\to \pm \infty }} \left\| u\left(t\right)-e^{it\Delta } u_{0}^{\pm } \right\| _{H^{s} (\R^{n} )} =0.
\end{equation}
\end{thm}

In Theorem \ref{thm 1.3.}, we didn't treat the case $s\notin \N$ and $n=3$. In this case, we have the following result.

\begin{thm}\label{thm 1.4.}
Let $1<s<\frac{3}{2} $, $0<b<{\rm 1}$, $\sigma =\frac{4-2b}{3-2s} $ and $0<\varepsilon <\min \left\{1-b,\; \frac{2s-1}{2} \right\}$. Assume that $f$ is of class ${\rm X} \left(\sigma ,s,b\right)$. Then for any $u_{0} \in H^{s} \left(\R^{3} \right)$, there exists $T=T\left(u_{0} \right)>0$ such that \eqref{GrindEQ__1_1_} has a unique solution
\begin{equation} \label{GrindEQ__1_14_}
u\in L^{\gamma \left(r\right)} \left(\left[-T,\;T\right],\;H_{r}^{s} \left(\R^{3} \right)\right),
\end{equation}
where $\left(\gamma \left(r\right),\;r\right)$ is an admissible pair satisfying
\begin{equation} \label{GrindEQ__1_15_}
\frac{1}{r} =\frac{1}{2} -\frac{5-2s+2\varepsilon }{6\left(\sigma +1\right)} .
\end{equation}
If $\left\| u_{0} \right\| _{\dot{H}^{s}(\R^{3})}$ is sufficiently small, then the above solution is global and scatters.
\end{thm}

The rest of this paper is organized as follows. In Section 2, we introduce some basic notation and recall some useful facts which are used in this paper. In Section 3, we establish the nonlinear estimates. In Section 4, we prove Theorem \ref{thm 1.3.} and \ref{thm 1.4.}.

\noindent
\section{Preliminaries}
First of all, let us introduce some basic notation. In this paper, $\C$, $\R$ and $\N$ will stand for the sets of complex, real and natural numbers, respectively. $C\left(>0\right)$ stands for the universal constant, which can be different at different places. We denote $a\lesssim b$ if $a\le Cb$ for some constant $C>0$. In addition, we write $a\sim b$ if $a\lesssim b \lesssim a$. $F$ denotes the Fourier transform; $F^{-1} $ denotes the inverse Fourier transform. We denote by $p'$ the dual number of $p\in [1,\;\infty ]$, i.e. $1/p+1/p'=1$. For $s\in \R$, we denote by $\left[s\right]$ the largest integer which is less than or equals to $s$ and by $\left\lceil s\right\rceil $ the minimal integer which is larger than or equals to $s$. For a multi-index $\alpha =\left(\alpha _{1} ,\;\alpha _{2} \;,\;\ldots ,\;\alpha _{n} \right)$, denote
\[D^{\alpha } =\partial _{x_{1} }^{\alpha _{1} } \cdots \partial _{x_{n} }^{\alpha _{n} },~ \left|\alpha \right|=\left|\alpha _{1} \right|+\cdots +\;\left|\alpha _{n} \right|.\]
For a function $f(z)$ defined for a complex variable $z$ and for a positive integer $k$, $k$-th order complex derivative of $f(z)$ is defined by
\[f^{(k)} (z)=\left(\frac{\partial ^{k} f}{\partial z^{k} } ,\; \frac{\partial ^{k} f}{\partial z^{k-1} \partial \bar{z}} ,\; ...,\; \frac{\partial ^{k} f}{\partial z\partial \bar{z}^{k-1} } ,\; \frac{\partial ^{k} f}{\partial \bar{z}^{k} } \right),\]
where
\[\frac{\partial f}{\partial z} =\frac{1}{2} \left(\frac{\partial f}{\partial x} -i\frac{\partial f}{\partial y} \right),\; \frac{\partial f}{\partial \bar{z}} =\frac{1}{2} \left(\frac{\partial f}{\partial x} +i\frac{\partial f}{\partial y} \right).\]
We also define its norm as
\[\left|f^{\left(k\right)} (z)\right|=\sum _{i=0}^{k}\left|\frac{\partial ^{k} f}{\partial z^{k-i} \partial \bar{z}^{i} } \right| .\]
As in \cite{G08}, for $0<p,\; q\le \infty $, we denote by $L^{p}(\R^{n} )$ and $L^{p,q} (\R^{n} )$ the Lebesgue space and Lorentz space, respectively. As in \cite{WHHG11}, for $s\in \R$ and $1<p<\infty $, we define the norms of nonhomogeneous Sobolev space $H_{p}^{s}(\R^{n} )$ and homogeneous Sobolev space $\dot{H}_{p}^{s} (\R^{n} )$, respectively, by
\[{\rm \; \; }\left\| f\right\| _{H_{p}^{s} (\R^{n} )} =\left\| F^{-1} \left(1+\left|\xi \right|^{2} \right)^{\frac{s}{2} } Ff\right\| _{L^{p} (\R^{n} )} , ~\left\| f\right\| _{\dot{H}_{p}^{s} (\R^{n} )} =\left\| F^{-1} \left|\xi \right|^{s} Ff\right\| _{L^{p} (\R^{n} )} ,\]
We shall abbreviate $H_{2}^{s}(\R^{n} )$ and $\dot{H}_{2}^{s}(\R^{n} )$ as $H^{s}(\R^{n})$ and $\dot{H}^{s} (\R^{n} )$, respectively.
We shall also use the space-time mixed space $L^{\gamma } (I,\;X(\R^{n} ))$ whose norm is defined by
\[\left\|f\right\|_{L^{\gamma } (I,\;X(\R^{n} ))} =\left(\int _{I}\left\| f\right\| _{X(\R^{n} )}^{\gamma } dt \right)^{\frac{1}{\gamma } } ,\]
with the usual modification when $\gamma=\infty$, where $I\subset \R$ is an interval and $X(\R^{n} )$ is a normed space on $\R^{n}$. If there is no confusion, $\R^{n} $ will be omitted in various function spaces. For two normed spaces $X$ and $Y$, $X\subset Y$ means that the space $X$ is continuously embedded in the space $Y$, that is, there exists a constant $C\left(>0\right)$ such that $\left\| f\right\| _{Y} \le C\left\| f\right\| _{X} $ for all $f\in X$.

Next, we recall some useful facts and estimates.

\begin{lem}[\cite{AK21}]\label{lem 2.1.}
Let $s>0$, $1<p<\infty $ and $v=s-\left[s\right]$. Then we have
\[\left\| f\right\| _{\dot{H}_{p}^{s} (\R^{n})} \sim \sum _{\left|\alpha \right|=\left[s\right]}\left\| D^{\alpha } f\right\| _{\dot{H}_{p}^{v} (\R^{n})}  .\]
\end{lem}
The following lemma is the well-known fractional product rule. See, for example, \cite{CW91, AK21}.

\begin{lem}\label{lem 2.2.}
Let $s\ge 0$, $1<r,\;r_{2} ,\;p_{1} <\infty $, $1<r_{1} ,\;p_{2} \le \infty $. Assume that
\[\frac{1}{r} =\frac{1}{r_{i} } +\frac{1}{p_{i} }~(i=1,\;2).\]
Then we have
\[\left\| fg\right\| _{\dot{H}_{r}^{s} } \lesssim\left\| f\right\| _{r_{1} } \left\| g\right\| _{\dot{H}_{p_{1} }^{s} } +\left\| f\right\| _{\dot{H}_{r_{2} }^{s} } \left\| g\right\| _{p_{2} } .\]
\end{lem}

\begin{lem}[\cite{WHHG11}]\label{lem 2.3.}
Let $-\infty <s_{2} \le s_{1} <\infty $ and $1<p_{1} \le p_{2} <\infty $ with $s_{1} -\frac{n}{p_{1} } =s_{2} -\frac{n}{p_{2} } $. Then there holds the embedding $\dot{H}_{p_{1} }^{s_{1} } \subset \dot{H}_{p_{2} }^{s_{2} } $.
\end{lem}

\begin{lem}[Fractional Hardy inequality in Lorentz spaces, \cite{HYZ12}]\label{lem 2.4.}
Let $1<p<\infty$, $0<s<\frac{n}{p}$ and $1\le q \le \infty$. There holds
\[\left\| |x|^{-s} f\right\| _{L^{p,q} } \lesssim\left\|F^{-1} \left|\xi \right|^{s} Ff\right\| _{L^{p,q}}. \]
\end{lem}

Since $L^{p,p}=L^{p}$, we immediately have the following fractional Hardy inequality.
\begin{cor}\label{cor 2.5.}
Let $1<p<\infty$ and $0<s<\frac{n}{p} $. Then we have
\[\left\| |x|^{-s} f\right\| _{L^{p} } \lesssim\left\| f\right\| _{\dot{H}_{p}^{s} } \]
\end{cor}

We end this section with recalling the well-known Strichartz estimates. For example, see \cite{C03,WHHG11} and the references therein.

\begin{lem}[Strichartz estimates]\label{lem 2.6.}
Let $S(t)=e^{it\Delta } $. Then, for any admissible pairs $(\gamma(p),\;p)$ and $(\gamma (r),\;r)$, we have
\begin{equation} \label{GrindEQ__2_1_}
\left\| S(t)\phi \right\| _{L^{\gamma (p)} (\R,\;\dot{H}_{p}^{s} )} \lesssim\left\| \phi \right\| _{\dot{H}^{s} } ,
\end{equation}
\begin{equation} \label{GrindEQ__2_2_}
\left\| \int _{0}^{t}S(t-\tau )f(\tau )d\tau  \right\| _{L^{\gamma (p)} (\R,\;\dot{H}_{p}^{s} )} \lesssim\left\| f\right\| _{L^{\gamma (r)'} (\R,\;\dot{H}_{r'}^{s} )}.
\end{equation}
\end{lem}

\section{Nonlinear estimates }

In this section, we establish the nonlinear estimates. The main tool in establishing the nonlinear estimates is the fractional Hardy inequality given in Corollary \ref{cor 2.5.}.

We divide the study in two cases: $s\in \N$ and $s\notin \N$.

When $s\in \N$, we obtain the following nonlinear estimates.
\begin{lem}\label{lem 3.1.}
Let $1<p,\;r<\infty $, $s\in \N$ and $\frac{b}{s} <\sigma $. Assume that $f\in C^{\left\lceil s\right\rceil } $ satisfies following condition:
\begin{equation} \label{GrindEQ__3_1_}
\left|f^{\left(k\right)} (u)\right|\lesssim\left|u\right|^{\sigma +1-k} ,~0\le k\le \left\lceil s\right\rceil ,~\left\lceil s\right\rceil \le \sigma +1.
\end{equation}
Suppose that
\begin{equation} \label{GrindEQ__3_2_}
\frac{1}{p} =\sigma \left(\frac{1}{r} -\frac{s}{n} \right)+\frac{1}{r} +\frac{b}{n} ,~ \frac{1}{r} -\frac{s}{n} >0.
\end{equation}
Then we have
\begin{equation} \label{GrindEQ__3_3_}
\left\| |x|^{-b} f(u)\right\| _{\dot{H}_{p}^{s} } \lesssim\left\| u\right\| _{\dot{H}_{r}^{s} }^{\sigma +1} .
\end{equation}
\end{lem}
\begin{proof}
By Lemma \ref{lem 2.1.}, we have
\begin{equation}\nonumber
\left\| |x|^{-b} f(u)\right\| _{\dot{H}_{p}^{s} } =\sum _{\left|\alpha \right|=s}\left\| D^{\alpha } \left(|x|^{-b} f(u)\right)\right\| _{p}  =\sum _{\left|\alpha '\right|+\left|\alpha ''\right|=s}\left\| D^{\alpha '} \left(|x|^{-b} \right)D^{\alpha ''} f(u)\right\| _{p}  .
\end{equation}
So it suffices to show that
\begin{equation} \label{GrindEQ__3_4_}
\left\| D^{\alpha '} \left(|x|^{-b} \right)D^{\alpha ''} f(u)\right\| _{p} \lesssim\left\| u\right\| _{\dot{H}_{r}^{s} }^{\sigma +1},
\end{equation}
for any $\alpha ',\;\alpha ''$ satisfying $\left|\alpha '\right|+\left|\alpha ''\right|=s$. We divide the study in two cases: $\left|\alpha ''\right|=0$ and $\left|\alpha ''\right|\ne 0$.

\textbf{Case 1.} We consider the case $\left|\alpha ''\right|=0$, i.e. $\left|\alpha '\right|=s$. In view of \eqref{GrindEQ__3_2_}, we have $\frac{1}{p} =\frac{\sigma+1}{\rho _{1} }$, where
\begin{equation} \label{GrindEQ__3_5__}
\frac{1}{\rho _{1} }=\frac{1}{r} -\frac{1}{n} \left(s-\frac{b+s}{\sigma +1} \right).
\end{equation}
We can see that $s-\frac{b+s}{\sigma +1} >0$ if, and only if, $\sigma >\frac{b}{s} $. Thus using \eqref{GrindEQ__3_5__} and Lemma \ref{lem 2.3.}, we have the embedding $\dot{H}_{r}^{s} \subset \dot{H}_{\rho _{1} }^{\frac{b+s}{\sigma +1} } $. Since $\frac{n}{r} >s>\frac{b+s}{\sigma +1}$, we also have ${\frac{b+s}{\sigma +1}}<\frac{n}{\rho_{1}}$. Hence it follows from Corollary \ref{cor 2.5.} that

\begin{eqnarray}\begin{split}\nonumber
\left\| D^{\alpha '} \left(|x|^{-b} \right)f(u)\right\| _{p} &\lesssim\left\| |x|^{-b-s} \left|u\right|^{\sigma +1} \right\| _{p} =\left\| |x|^{-\frac{b+s}{\sigma +1} } u\right\| _{\rho _{1} }^{\sigma +1}  \\
&\lesssim\left\| u\right\| _{\dot{H}_{\rho _{1} }^{\frac{b+s}{\sigma +1}}}^{\sigma +1} \lesssim\left\| u\right\| _{\dot{H}_{r}^{s} }^{\sigma +1} .
\end{split}\end{eqnarray}

\textbf{Case 2.} We consider the case $\alpha '\ne \alpha $, i.e. we have to show that
\[\left\| D^{\alpha '} \left(|x|^{-b} \right)D^{\alpha ''} \left(f(u)\right)\right\| _{p} \lesssim\left\| u\right\| _{\dot{H}_{r}^{s} }^{\sigma +1} ,\]
where $\left|\alpha '\right|+\left|\alpha ''\right|=s$ and $\left|\alpha ''\right|\ge 1$. Without loss of generality and for simplicity, we assume that $f$ is a function of a real variable. It follows from the Leibniz rule of derivatives that
\begin{equation} \label{GrindEQ__3_6_}
D^{\alpha ''} (f(u))=\sum _{q=1}^{\left|\alpha ''\right|}\sum _{\Lambda _{\alpha ''}^{q} }C_{\alpha '',\;q} f^{\left(q\right)} (u)\prod _{i=1}^{q}D^{\alpha _{i} } u,
\end{equation}
where $\Lambda _{\alpha ''}^{q} =\left(\alpha _{1} +\cdots +\alpha _{q} =\alpha '',\;\left|\alpha _{i} \right|\ge 1\right)$. Thus we have
\begin{eqnarray}\begin{split}\nonumber
\left\| D^{\alpha '} \left(|x|^{-b} \right)D^{\alpha ''} \left(f(u)\right)\right\| _{p} &\lesssim\left\| |x|^{-b-\left|\alpha '\right|} \sum _{q=1}^{\left|\alpha ''\right|}\left|u\right|^{\sigma +1-q} \sum _{\Lambda _{\alpha ''}^{q} }\prod _{i=1}^{q}D^{\alpha _{i} }u\right\| _{p} \\
&\lesssim\sum _{q=1}^{\left|\alpha ''\right|}\sum _{\Lambda _{\alpha ''}^{q} }\left\| |x|^{-b-\left|\alpha '\right|} \left|u\right|^{\sigma +1-q} \prod _{i=1}^{q}D^{\alpha _{i} } u \right\| _{p}.
\end{split}\end{eqnarray}
So it suffices to show that
\begin{equation} \nonumber
I\equiv \left\| |x|^{-b-\left|\alpha '\right|} \left|u\right|^{\sigma +1-q} \prod _{i=1}^{q}D^{\alpha _{i} } u \right\| _{p} \lesssim\left\| u\right\| _{\dot{H}_{r}^{s} }^{\sigma +1} ,
\end{equation}
where $\alpha _{1} +\cdots +\alpha _{q} =\alpha ''$, $\left|\alpha _{i} \right|\ge 1$, $\left|\alpha '\right|+\left|\alpha ''\right|=s$ and $\left|\alpha ''\right|\ge q\ge 1$.

\emph{Case 2.1.} First, we consider the case $\left|\alpha ''\right|=s$ and $q=1$, i.e. we have to show that
$$I=\left\| |x|^{-b} \left|u\right|^{\sigma } D^{\alpha ''} u\right\| _{p} \lesssim\left\| u\right\| _{\dot{H}_{r}^{s} }^{\sigma +1} ,
$$
where $\left|\alpha ''\right|=s$. In fact, it follows from \eqref{GrindEQ__3_2_} that
\begin{equation} \nonumber
\sigma \left(\frac{1}{r} -\frac{1}{n} \left(s-\frac{b}{\sigma } \right)\right)+\frac{1}{r} =\frac{1}{p} .
\end{equation}
Putting $\frac{1}{\rho _{2} } :=\frac{1}{r} -\frac{1}{n} \left(s-\frac{b}{\sigma } \right)$, we have the embedding $\dot{H}_{r}^{s} \subset \dot{H}_{\rho _{2} }^{b/\sigma } $. Hence, noticing $\frac{n}{r} >s>\frac{b}{\sigma } $, it follows from H\"{o}lder inequality and Corollary \ref{cor 2.5.} that
\begin{equation} \nonumber
I=\left\| |x|^{-b} \left|u\right|^{\sigma } D^{\alpha } u\right\| _{p} \le \left\| |x|^{-\frac{b}{\sigma } } u\right\| _{\rho _{2} }^{\sigma } \left\| D^{\alpha } u\right\| _{r} \lesssim\left\| u\right\| _{\dot{H}_{\rho _{2} }^{b/\sigma } }^{\sigma } \left\| u\right\| _{\dot{H}_{r}^{s} } \lesssim\left\| u\right\| _{\dot{H}_{r}^{s} }^{\sigma +1} .
\end{equation}

\emph{Case 2.2.} Next, we consider the case $s>\left|\alpha ''\right|\ge 1$ or $q\ge 2$. In this case, we can see that $s-\left|\alpha _{i} \right|\ge 1$, for $1\le i\le q$. Since $\sigma >\frac{b}{s} $, we have
\begin{equation} \label{GrindEQ__3_7_}
b-sq+s<s\left(\sigma +1-q\right).
\end{equation}
Putting $c:=b-sq+s$, we can see that
\begin{equation} \label{GrindEQ__3_8_}
b+\left|\alpha '\right|=c+\sum _{i=1}^{q}\left(s-\left|\alpha _{i} \right|\right) .
\end{equation}
We also divide the study in two cases: $c\ge0$ and $c<0$.

$\cdot$ If $c\ge 0$, it follows from \eqref{GrindEQ__3_7_} that $\sigma +1-q>0$.
       Using \eqref{GrindEQ__3_2_} and \eqref{GrindEQ__3_8_}, we have
\begin{equation} \nonumber
\left(\sigma +1-q\right)\left(\frac{1}{r} -\frac{1}{n} \left(s-\frac{c}{\sigma +1-q} \right)\right)+\frac{q}{r} =\frac{1}{p} .
\end{equation}
Putting $\frac{1}{\rho_{3} } :=\frac{1}{r} -\frac{1}{n} \left(s-\frac{c}{\sigma +1-q} \right)$,
and noticing $\frac{n}{r} >s>\frac{c}{\sigma +1-q} $, it follows from H\"{o}lder inequality and Corollary \ref{cor 2.5.} that
\begin{eqnarray}\begin{split} \nonumber
I&=\left\| |x|^{-b-\left|\alpha '\right|} \left|u\right|^{\sigma +1-q} \prod _{i=1}^{q}D^{\alpha _{i} } u \right\| _{p}\\
&\lesssim\left\| |x|^{-c} \left|u\right|^{\sigma +1-q} \prod _{i=1}^{q}|x|^{-\left(s-\left|\alpha _{i} \right|\right)} D^{\alpha _{i} } u \right\| _{p}  \\
&\lesssim\left\| \left(|x|^{-\frac{c}{\sigma +1-q} } \left|u\right|\right)^{\sigma +1-q} \prod _{i=1}^{q}|x|^{-\left(s-\left|\alpha _{i} \right|\right)} D^{\alpha _{i} } u \right\| _{p} \\
&\lesssim\left\| |x|^{-\frac{c}{\sigma +1-q} } u\right\| _{\rho _{3} }^{\sigma +1-q} \prod _{i=1}^{q}\left\| |x|^{-\left(s-\left|\alpha _{i} \right|\right)} D^{\alpha _{i} } u\right\| _{r}   \\
&\lesssim\left\| u\right\| _{\dot{H}_{\rho _{3} }^{\frac{c}{\sigma +1-q} } }^{\sigma +1-q} \prod _{i=1}^{q}\left\| D^{\alpha _{i} } u\right\| _{\dot{H}_{r}^{s-\left|\alpha _{i} \right|} }  \lesssim\left\| u\right\| _{\dot{H}_{r}^{s} }^{\sigma +1} ,
\end{split}\end{eqnarray}
where the last inequality follows from $\dot{H}_{r}^{s} \subset \dot{H}_{\rho _{3} }^{\frac{c}{\sigma +1-q} } $.

$\cdot$ If $c<0$, then we have
\begin{equation} \nonumber
c=b+\left|\alpha '\right|-\sum _{i=1}^{q}\left(s-\left|\alpha _{i} \right|\right) <0.
\end{equation}
Thus we can take a minimal integer $j\ge 1$ such that
\begin{equation} \nonumber
b+\left|\alpha '\right|-\sum _{i=1}^{j}\left(s-\left|\alpha _{i} \right|\right) \le 0,
\end{equation}
indeed, we have
\begin{equation} \nonumber
\sum _{i=1}^{j-1}\left(s-\left|\alpha _{i} \right|\right) <b+\left|\alpha '\right|\le \sum _{i=1}^{j}\left(s-\left|\alpha _{i} \right|\right),
\end{equation}
where we assume that $\sum _{i=1}^{0}a_{i}=0$. Thus there exists $s_{0}$ such that
\begin{equation} \label{GrindEQ__3_9_}
b+\left|\alpha '\right|=\sum _{i=1}^{j-1}\left(s-\left|\alpha _{i} \right|\right) +\left(s_{0} -\left|\alpha _{j} \right|\right)
\end{equation}
and $s\ge s_{0} >\left|\alpha _{j} \right|$. Hence we have
\begin{equation} \label{GrindEQ__3_10_}
I=\left\| \left|u\right|^{\sigma +1-q} |x|^{-\left(s_{0} -\left|\alpha _{j} \right|\right)} D^{\alpha _{j} } u\prod _{i=1}^{j-1}|x|^{-\left(s-\left|\alpha _{i} \right|\right)} D^{\alpha _{i} } u \prod _{k=j+1}^{q}D^{\alpha _{k} } u \right\| _{p},
\end{equation}
where we assume that $\prod _{i=1}^{0}a_{i}=1$. It also follows from \eqref{GrindEQ__3_2_} and \eqref{GrindEQ__3_9_} that
\begin{equation} \label{GrindEQ__3_11_}
\frac{1}{p} =\left(\sigma +1-q\right)\left(\frac{1}{r} -\frac{s}{n} \right)+\frac{j-1}{r} +\sum _{i=j+1}^{q}\left(\frac{1}{r} -\frac{s-\left|\alpha _{i} \right|}{n} \right) +\frac{1}{r} -\frac{s-s_{0} }{n} .
\end{equation}
Putting
\begin{equation} \nonumber
\frac{1}{a_{0} } :=\frac{1}{r} -\frac{s}{n},~ \frac{1}{a_{i} } :=\frac{1}{r} -\frac{s-\left|\alpha _{i} \right|}{n} , ~\frac{1}{a} :=\frac{1}{r} -\frac{s-s_{0} }{n} ,
\end{equation}
and using \eqref{GrindEQ__3_10_}, \eqref{GrindEQ__3_11_}, H\"{o}lder inequality and Corollary \ref{cor 2.5.}, we have
\begin{eqnarray}\begin{split} \nonumber
I&\lesssim\left\| u\right\| _{a_{0} }^{\sigma +1-q} \left\| |x|^{-\left(s_{0} -\left|\alpha _{j} \right|\right)} D^{\alpha _{j} } u\right\| _{a} \prod _{i=1}^{j-1}\left\| |x|^{-\left(s-\left|\alpha _{i} \right|\right)} D^{\alpha _{i} } u\right\| _{r}  \prod _{i=j+1}^{q}\left\| D^{\alpha _{i} } u\right\| _{a_{i} } \\
&\lesssim\left\| u\right\| _{\dot{H}_{r}^{s} }^{\sigma +1-q} \left\| D^{\alpha _{j} } u\right\| _{\dot{H}_{a}^{s_{0} -\left|\alpha _{j} \right|} } \prod _{i=1}^{j-1}\left\| D^{\alpha _{i} } u\right\| _{\dot{H}_{r}^{s-\left|\alpha _{i} \right|} }  \prod _{i=j+1}^{q}\left\| u\right\| _{\dot{H}_{a_{i} }^{\left|\alpha _{i} \right|} } \\
&\lesssim\left\| u\right\| _{\dot{H}_{r}^{s} }^{\sigma } \left\| u\right\| _{\dot{H}_{a}^{s_{0} } } \lesssim\left\| u\right\| _{\dot{H}_{r}^{s} }^{\sigma +1} ,
\end{split}\end{eqnarray}
where we use the embeddings: $\dot{H}_{r}^{s} \subset L^{a_{0} } $, $\dot{H}_{r}^{s} \subset \dot{H}_{a}^{s_{0} } $ and $\dot{H}_{r}^{s} \subset \dot{H}_{a_{i} }^{\left|\alpha _{i} \right|} $.
\end{proof}

If $s\notin \N$, we have the following similar nonlinear estimates.

\begin{lem}\label{lem 3.2.}
Let $1<p,\;r<\infty $, $s>1$, $s\notin \N$ and $\frac{b+1}{s} <\sigma $. Assume further that $\frac{1}{p}+\frac{\left\lceil s\right\rceil-s}{n}<1$.
If \eqref{GrindEQ__3_1_} and \eqref{GrindEQ__3_2_} are satisfied, then \eqref{GrindEQ__3_3_} still holds.
\end{lem}
\begin{proof}
By Lemma \ref{lem 2.1.}, we have
\begin{eqnarray}\begin{split} \nonumber
\left\| |x|^{-b} f(u)\right\| _{\dot{H}_{p}^{s} } &=\sum _{\left|\alpha \right|=\left[s\right]}\left\| D^{\alpha } \left(|x|^{-b} f(u)\right)\right\| _{\dot{H}_{p}^{v} } \\
&=\sum _{\left|\alpha '\right|+\left|\alpha ''\right|=\left[s\right]}\left\| D^{\alpha '} \left(|x|^{-b} \right)D^{\alpha ''} f(u)\right\| _{\dot{H}_{p}^{v} }  ,
\end{split}\end{eqnarray}
where $v=s-\left[s\right]$. So it suffices to show that
\begin{equation} \nonumber
II\equiv \left\| D^{\alpha '} \left(|x|^{-b} \right)D^{\alpha ''} \left(f(u)\right)\right\| _{\dot{H}_{p}^{v} } \lesssim\left\| u\right\| _{\dot{H}_{r}^{s} }^{\sigma +1},
\end{equation}
where $\left|\alpha '\right|+\left|\alpha ''\right|=\left[s\right]$. We also divide study in two cases: $\left|\alpha ''\right|=0$ and $\left|\alpha ''\right|\ne 0$.

\textbf{Case 1.} We consider the case $\left|\alpha ''\right|=0$, i.e. $\left|\alpha '\right|=\left[s\right]$. Since $\frac{1}{p}+\frac{\left\lceil s\right\rceil-s}{n}<1$, putting
\begin{equation} \nonumber
\frac{1}{p_{1} } :=\frac{1}{p} -\frac{v}{n} +\frac{1}{n} ,
\end{equation}
we have $p_{1} >1$. Furthermore, there holds the embedding $\dot{H}_{p_{1} }^{1} \subset \dot{H}_{p}^{v} $. Thus we have
\begin{eqnarray}\begin{split} \nonumber
II&=\left\| D^{\alpha '} \left(|x|^{-b} \right)f(u)\right\| _{\dot{H}_{p}^{v} } \lesssim\left\| D^{\alpha '} \left(|x|^{-b} \right)f(u)\right\| _{\dot{H}_{p_{1} }^{1} } \\
&=\sum _{i=1}^{n}\left\| \partial _{x_{i} } \left(D^{\alpha '} \left(|x|^{-b} \right)f(u)\right)\right\| _{p_{1} }.
\end{split}\end{eqnarray}
So it suffices to show that
\begin{equation} \nonumber
\left\| \partial _{x_{i} } \left(D^{\alpha '} \left(|x|^{-b} \right)f(u)\right)\right\| _{p_{1} } \lesssim\left\| u\right\| _{\dot{H}_{r}^{s} }^{\sigma +1},
\end{equation}
for any $1\le i\le n$. It also follows from \eqref{GrindEQ__3_2_} that
\begin{equation} \label{GrindEQ__3_12_}
\frac{1}{p_{1} } =\sigma \left(\frac{1}{r} -\frac{s}{n} \right)+\frac{1}{r} +\frac{b}{n} +\frac{1}{n} -\frac{v}{n} .
\end{equation}
Using \eqref{GrindEQ__3_1_} and \eqref{GrindEQ__3_12_}, we have
\begin{eqnarray}\begin{split} \label{GrindEQ__3_13_}
&\left\| \partial _{x_{i} } \left(D^{\alpha '} \left(|x|^{-b} \right)f(u)\right)\right\| _{p_{1} }\\
&~~~~~~~~\le \left\| \partial _{x_{i} } \left(D^{\alpha '} \left(|x|^{-b} \right)\right)f(u)\right\| _{p_{1} } +\left\| D^{\alpha '} \left(|x|^{-b} \right)\partial _{x_{i} } \left(f(u)\right)\right\| _{p_{1} } \\
&~~~~~~~~\lesssim\left\| |x|^{-b-\left[s\right]-1} \left|u\right|^{\sigma } u\right\| _{p_{1} } +\left\| |x|^{-b-\left[s\right]} \left|u\right|^{\sigma } \partial _{x_{i} } u\right\| _{p_{1} } \equiv II_{1} +II_{2}
\end{split}\end{eqnarray}

First, we estimate $II_{1} $. In view of \eqref{GrindEQ__3_12_}, we have
\begin{equation} \nonumber
\frac{1}{p_{1} } =\sigma \left(\frac{1}{r} -\frac{1}{n} \left(s-\frac{b+1}{\sigma } \right)\right)+\frac{1}{r} -\frac{v}{n} ,
\end{equation}
Putting
\begin{equation} \nonumber
\frac{1}{a_{1} } :=\frac{1}{r} -\frac{1}{n} \left(s-\frac{b+1}{\sigma } \right), ~\frac{1}{b_{1} } :=\frac{1}{r} -\frac{v}{n} ,
\end{equation}
and using the same argument as in the proof of Lemma \ref{lem 3.1.}, we have
\begin{eqnarray}\begin{split} \nonumber
II_{1}&=\left\| \left(|x|^{-\frac{b+1}{\sigma } } \left|u\right|\right)^{\sigma } |x|^{-\left[s\right]} u\right\| _{p_{1} } \le \left\| |x|^{-\frac{b+1}{\sigma } } u\right\| _{a_{1} }^{\sigma } \left\| |x|^{-\left[s\right]} u\right\| _{b_{1} } \\
&\lesssim\left\| u\right\| _{\dot{H}_{a_{1} }^{\frac{b+1}{\sigma } } }^{\sigma } \left\| u\right\| _{\dot{H}_{b_{1} }^{[s]} } \lesssim\left\| u\right\| _{\dot{H}_{r}^{s} }^{\sigma +1}.
\end{split}\end{eqnarray}

Next, we estimate $II_{2} $. Using the same argument as in the estimate of $II_{1} $, we also have
\begin{eqnarray}\begin{split} \nonumber
II_{2} &=\left\| |x|^{-b-1} \left|u\right|^{\sigma } |x|^{-\left(\left[s\right]-1\right)} \partial _{x_{i} } u\right\| _{p_{1} } \le \left\| |x|^{-\frac{b+1}{\sigma } } u\right\| _{a_{1} }^{\sigma } \left\| |x|^{-\left(\left[s\right]-1\right)} \partial _{x_{i} } u\right\| _{b_{1} } \\
&\lesssim\left\| u\right\| _{\dot{H}_{r}^{s} }^{\sigma } \left\| \partial _{x_{i} } u\right\| _{\dot{H}_{b_{1} }^{[s]-1} } \lesssim\left\| u\right\| _{\dot{H}_{r}^{s} }^{\sigma +1}.
\end{split}\end{eqnarray}

\textbf{Case 2.} We consider the case $\left|\alpha ''\right|\ne 0$, i.e. we have to show that
\begin{equation} \nonumber
II=\left\| D^{\alpha '} \left(|x|^{-b} \right)D^{\alpha ''} \left(f(u)\right)\right\| _{\dot{H}_{p}^{v} } \lesssim\left\| u\right\| _{\dot{H}_{r}^{s} }^{\sigma +1} ,
\end{equation}
where $\left|\alpha '\right|+\left|\alpha ''\right|=\left[s\right]$ and $\left|\alpha ''\right|\ge 1$. It follows from \eqref{GrindEQ__3_6_} that
\begin{equation} \nonumber
II=\left\| D^{\alpha '} \left(|x|^{-b} \right)D^{\alpha ''} \left(f(u)\right)\right\| _{\dot{H}_{p}^{v} } \lesssim\sum _{q=1}^{\left|\alpha ''\right|}\sum _{\Lambda _{\alpha ''}^{q} }\left\| D^{\alpha '} \left(|x|^{-b} \right)f^{(q)} (u)\prod _{i=1}^{q}D^{\alpha _{i} } u \right\| _{\dot{H}_{p}^{v} },
\end{equation}
where $\Lambda _{\alpha ''}^{q} =\left(\alpha _{1} +\cdots +\alpha _{q} =\alpha '',\;\left|\alpha _{i} \right|\ge 1\right)$. Hence it suffices to show that
\begin{equation} \label{GrindEQ__3_14_}
II_{3} \equiv \left\| D^{\alpha '} \left(|x|^{-b} \right)f^{(q)} (u)\prod _{i=1}^{q}D^{\alpha _{i} } u \right\| _{\dot{H}_{p}^{v} } \lesssim\left\| u\right\| _{\dot{H}_{r}^{s} }^{\sigma +1} ,
\end{equation}
where $\alpha _{1} +\cdots +\alpha _{q} =\alpha ''$, $\left|\alpha _{i} \right|\ge 1$, $\left|\alpha '\right|+\left|\alpha ''\right|=\left[s\right]$ and $\left|\alpha ''\right|\ge q\ge 1$.

\noindent We divide the proof of \eqref{GrindEQ__3_14_} in two cases.

\emph{Case 2.1.} We consider the case $\left|\alpha ''\right|=\left[s\right]$ and $q=1$, i.e. we have to prove that
\begin{equation} \nonumber
II_{3} \equiv \left\| |x|^{-b} f'(u)D^{\alpha ''} u\right\| _{\dot{H}_{p}^{v} } \lesssim\left\| u\right\| _{\dot{H}_{r}^{s} }^{\sigma +1} ,
\end{equation}
where $\left|\alpha ''\right|=\left[s\right]$. Putting
\begin{equation} \label{GrindEQ__3_15_}
\frac{1}{p_{2} }:=\sigma \left(\frac{1}{r} -\frac{s}{n} \right)+\frac{b}{n} ,~ \frac{1}{p_{3} } :=\frac{1}{p_{2} } +\frac{v}{n} ,~ \frac{1}{r_{3} }:=\frac{1}{r} -\frac{v}{n} ,
\end{equation}
it follows from \eqref{GrindEQ__3_2_} and Lemma \ref{lem 2.2.} that
\begin{equation} \label{GrindEQ__3_46_}
II_{3} =\left\| |x|^{-b} f'(u)\right\| _{p_{2} } \left\| D^{\alpha ''} u\right\| _{\dot{H}_{r}^{v} } +\left\| |x|^{-b} f'(u)\right\| _{\dot{H}_{p_{3} }^{v} } \left\| D^{\alpha ''} u\right\| _{r_{3} } \equiv II_{4} +II_{5} .
\end{equation}

First, we estimate $II_{4} $. Putting $\frac{1}{a_{2} } :=\frac{1}{r} -\frac{1}{n} \left(s-\frac{b}{\sigma } \right)$, it follows from \eqref{GrindEQ__3_15_} and Corollary \ref{cor 2.5.} that
\begin{eqnarray}\begin{split} \nonumber
II_{4}&\lesssim\left\| |x|^{-b} \left|u\right|^{\sigma } \right\| _{p_{2} } \left\| D^{\alpha ''} u\right\| _{\dot{H}_{r}^{v} } \lesssim\left\| |x|^{-\frac{b}{\sigma } } u\right\| _{a_{2} }^{\sigma } \left\| u\right\| _{\dot{H}_{r}^{s} } \\
&\lesssim\left\| u\right\| _{\dot{H}_{a_{2} }^{\frac{b}{\sigma } } }^{\sigma } \left\| u\right\| _{\dot{H}_{r}^{s} } \lesssim\left\| u\right\| _{\dot{H}_{r}^{s} }^{\sigma +1} .
\end{split}\end{eqnarray}

Next, we estimate $II_{5} $. Using \eqref{GrindEQ__3_15_}, we have the embedding $\dot{H}_{r}^{v} \subset L^{r_{3} } $, which implies that
\begin{equation} \label{GrindEQ__3_17_}
\left\| D^{\alpha ''} u\right\| _{r_{3} } \lesssim\left\| D^{\alpha ''} u\right\| _{\dot{H}_{r}^{v} } \lesssim\left\| u\right\| _{\dot{H}_{r}^{s} } .
\end{equation}
Putting $\frac{1}{p_{4} } :=\frac{1}{p_{3} } -\frac{v}{n} +\frac{1}{n} $, we have the embedding $\dot{H}_{p_{4} }^{1} \subset \dot{H}_{p_{3} }^{v} $. Thus we have
\begin{equation} \label{GrindEQ__3_18_}
\left\| |x|^{-b} f'(u)\right\| _{\dot{H}_{p_{3} }^{v} } \lesssim\left\| |x|^{-b} f'(u)\right\| _{\dot{H}_{p_{4} }^{1} } =\sum _{i=1}^{n}\left\| \partial _{x_{i} } \left(|x|^{-b} f'(u)\right)\right\| _{p_{4} }  .
\end{equation}
Noticing that
\begin{equation}\nonumber
\frac{1}{p_{4} } =\sigma \left(\frac{1}{r} -\frac{s}{n} \right)+\frac{b+1}{n} =\frac{\sigma }{a_{1} } , ~\frac{1}{a_{1} } =\frac{1}{r} -\frac{1}{n} \left(s-\frac{b+1}{\sigma } \right),
\end{equation}
we have
\begin{eqnarray}\begin{split} \label{GrindEQ__3_19_}
\left\| \partial _{x_{i} } \left(|x|^{-b} \right)f'(u)\right\| _{p_{4} } &\lesssim\left\| |x|^{-b-1} \left|u\right|^{\sigma } \right\| _{p_{4} } =\left\| |x|^{-\frac{b+1}{\sigma } } u\right\| _{a_{1} }^{\sigma } \\
&\lesssim\left\| u\right\| _{\dot{H}_{a_{1} }^{\frac{b+1}{\sigma } } }^{\sigma } \lesssim\left\| u\right\| _{\dot{H}_{r}^{s} }^{\sigma } .
\end{split}\end{eqnarray}
On the other hand, it follows from Leibniz rule of derivatives that
\begin{equation} \label{GrindEQ__3_20_}
\left\| |x|^{-b} \partial _{x_{i} } \left(f'(u)\right)\right\| _{p_{4} } \lesssim\left\| |x|^{-b} \left|u\right|^{\sigma -1} \partial _{x_{i} } u\right\| _{p_{4} }.
\end{equation}
Putting $c_{1} :=b+1-s$ and noticing $\frac{b+1}{s} <\sigma $, we can see that $c_{1} <s\left(\sigma -1\right)$.

We divide the cases into $c_{1} \ge 0$ and $c_{1} <0$.

$\cdot$ If $c_{1} \ge 0$, we have $\sigma -1>\frac{c_{1} }{s} \ge 0$. Noticing that
\begin{equation} \nonumber
\frac{1}{p_{4} } =\left(\sigma -1\right)\left(\frac{1}{r} -\frac{1}{n} \left(s-\frac{c_{1} }{\sigma -1} \right)\right)+\frac{1}{r},
\end{equation}
and putting $\frac{1}{a_{3} } :=\frac{1}{r} -\frac{1}{n} \left(s-\frac{c_{1} }{\sigma -1} \right)$, we have
\begin{eqnarray}\begin{split} \label{GrindEQ__3_22_}
\left\| |x|^{-b} \left|u\right|^{\sigma -1} \partial _{x_{i} } u\right\| _{p_{4} } &=\left\| |x|^{-\left(b-s+1\right)} \left|u\right|^{\sigma -1} |x|^{-\left(s-1\right)} \partial _{x_{i} } u\right\| _{p_{4} } \\
&\le \left\| |x|^{-\frac{c_{1} }{\sigma -1} } u\right\| _{a_{3} }^{\sigma-1 } \left\| |x|^{-\left(s-1\right)} \partial _{x_{i} } u\right\| _{r} \\
&\lesssim\left\| u\right\| _{\dot{H}_{a_{3} }^{\frac{c_{1} }{\sigma -1} } }^{\sigma-1 } \left\| \partial _{x_{i} } u\right\| _{\dot{H}_{r}^{s-1} } \lesssim\left\| u\right\| _{\dot{H}_{r}^{s} }^{\sigma} .
\end{split}\end{eqnarray}

$\cdot$ Next, we consider the case $c_{1} <0$, i.e. $b+1-s<0$. Noticing
\begin{equation} \nonumber
\frac{1}{p_{4} } =\left(\sigma -1\right)\left(\frac{1}{r} -\frac{s}{n} \right)+\frac{1}{r} -\frac{s-b-1}{n} ,
\end{equation}
and $s-b-1>0$, we have
\begin{eqnarray}\begin{split} \label{GrindEQ__3_22_}
\left\| |x|^{-b} \left|u\right|^{\sigma -1} \partial _{x_{i} } u\right\| _{p_{4} } &=\left\| \left|u\right|^{\sigma -1} |x|^{-b} \partial _{x_{i} } u\right\| _{p_{4} } \;\le \left\| u\right\| _{\dot{H}_{r}^{s} }^{\sigma -1} \left\| |x|^{-b} \partial _{x_{i} } u\right\| _{a_{4} }  \\
&\lesssim\left\| u\right\| _{\dot{H}_{r}^{s} }^{\sigma-1} \left\| \partial _{x_{i} } u\right\| _{\dot{H}_{a_{4} }^{b} } \lesssim\left\| u\right\| _{\dot{H}_{r}^{s} }^{\sigma},
\end{split}\end{eqnarray}
where $\frac{1}{a_{4}}=\frac{1}{r} -\frac{s-b-1}{n}$. In view of \eqref{GrindEQ__3_20_}-\eqref{GrindEQ__3_22_}, we have
\begin{equation} \label{GrindEQ__3_23_}
\left\| |x|^{-b} \partial _{x_{i} } \left(f'(u)\right)\right\| _{p_{4} } \lesssim\left\| u\right\| _{\dot{H}_{r}^{s} }^{\sigma} .
\end{equation}
Using \eqref{GrindEQ__3_18_}, \eqref{GrindEQ__3_19_} and \eqref{GrindEQ__3_23_}, we immediately have
\begin{equation} \label{GrindEQ__3_24_}
\left\| |x|^{-b} f'(u)\right\| _{\dot{H}_{p_{3} }^{v} } \lesssim\left\| u\right\| _{\dot{H}_{r}^{s} }^{\sigma} .
\end{equation}
In view of \eqref{GrindEQ__3_17_} and \eqref{GrindEQ__3_24_}, we have
\begin{equation} \nonumber
II_{5} =\left\| |x|^{-b} f'(u)\right\| _{\dot{H}_{p_{3} }^{v} } \left\| D^{\alpha ''} u\right\| _{r_{3} } \lesssim\left\| u\right\| _{\dot{H}_{r}^{s} }^{\sigma +1} ,
\end{equation}
this completes the proof of \eqref{GrindEQ__3_14_} in Case 2.1.

\emph{Case 2.2.} Now, we prove \eqref{GrindEQ__3_14_} in the case $\left[s\right]>\left|\alpha ''\right|\ge 1$ or $q\ge 2$.

In this case, we can see that $\left[s\right]\ge \left|\alpha _{i} \right|+1$ for $1\le i\le q$. Using the embedding $\dot{H}_{p_{1} }^{1} \subset \dot{H}_{p}^{v} $, we have
\begin{eqnarray}\begin{split} \label{GrindEQ__3_25_}
II_{3}&=\left\| D^{\alpha '} \left(|x|^{-b} \right)f^{(q)} (u)\prod _{i=1}^{q}D^{\alpha _{i} } u \right\| _{\dot{H}_{p_{1} }^{1} } \\
&=\sum _{k=1}^{n}\left\| \partial _{x_{k} } \left(D^{\alpha '} \left(|x|^{-b} \right)f^{(q)} (u)\prod _{i=1}^{q}D^{\alpha _{i} } u \right)\right\| _{p_{1} },
\end{split}\end{eqnarray}
where $\frac{1}{p_{1} }=\frac{1}{p} -\frac{v}{n} +\frac{1}{n}$. We can see that
\begin{equation} \label{GrindEQ__3_26_}
\left\| \partial _{x_{k} } \left(D^{\alpha '} \left(|x|^{-b} \right)f^{(q)} (u)\prod _{i=1}^{q}D^{\alpha _{i} } u \right)\right\| _{p_{1} } \le II_{6} +II_{7} +II_{8},
\end{equation}
where
\begin{equation} \label{GrindEQ__3_27_}
II_{6} =\left\| \partial _{x_{k} } \left(D^{\alpha '} \left(|x|^{-b} \right)\right)f^{(q)} (u)\prod _{i=1}^{q}D^{\alpha _{i} } u \right\| _{p_{1} } ,
\end{equation}
\begin{equation} \label{GrindEQ__3_28_}
II_{7} =\left\| D^{\alpha '} \left(|x|^{-b} \right)\partial _{x_{k} } \left(f^{(q)} (u)\right)\prod _{i=1}^{q}D^{\alpha _{i} } u \right\| _{p_{1} } ,
\end{equation}
\begin{equation} \label{GrindEQ__3_29_}
II_{8} =\left\| D^{\alpha '} \left(|x|^{-b} \right)f^{(q)} (u)\partial _{x_{k} } \left(\prod _{i=1}^{q}D^{\alpha _{i} } u \right)\right\| _{p_{1} } .
\end{equation}

First, we estimate $II_{6} $. It follows from \eqref{GrindEQ__3_2_} that
\begin{equation} \label{GrindEQ__3_30_}
\frac{1}{p_{1} } =\left(\sigma +1-q\right)\left(\frac{1}{r} -\frac{s}{n} \right)+\frac{q}{r} +\frac{b+1+\left[s\right]-qs}{n} .
\end{equation}
Since $b+1<\sigma s$, we have $b+\left[s\right]+1<\sigma s+s$ which is equivalent to
\begin{equation} \label{GrindEQ__3_31_}
b+\left[s\right]+1-qs<\left(\sigma +1-q\right)s
\end{equation}
Putting $c_{2} :=b+\left[s\right]+1-qs$, we have
\begin{equation} \label{GrindEQ__3_32_}
b+\left|\alpha '\right|+1=\sum _{i=1}^{q}\left(s-\left|\alpha _{i} \right|\right) +c_{2} .
\end{equation}
Using \eqref{GrindEQ__3_30_}-\eqref{GrindEQ__3_32_} and the argument similar to that used in Case 2.2 in the proof of Lemma \ref{lem 3.1.}, we can prove
\begin{equation} \label{GrindEQ__3_33_}
II_{6} \lesssim\left\| |x|^{-b-\left|\alpha '\right|-1} \left|u\right|^{\sigma +1-q} \prod _{i=1}^{q}D^{\alpha _{i} } u \right\| _{p_{1} } \lesssim\left\| u\right\| _{\dot{H}_{r}^{s} }^{\sigma +1} ,
\end{equation}
whose proof will be omitted.

Next, we estimate $II_{7} $. Putting $c_{3} :=c_{2} -1$, it follows from \eqref{GrindEQ__3_30_}-\eqref{GrindEQ__3_32_} that
\begin{equation} \label{GrindEQ__3_34_}
\frac{1}{p_{1} } =\left(\sigma -q\right)\left(\frac{1}{r} -\frac{s}{n} \right)+\frac{q}{r} +\frac{c_{3} }{n} +\frac{1}{r} -\frac{s-1}{n} ,
\end{equation}
\begin{equation} \label{GrindEQ__3_35_}
c_{3} <\left(\sigma -q\right)s+s-1,
\end{equation}
\begin{equation} \label{GrindEQ__3_36_}
b+\left|\alpha '\right|=\sum _{i=1}^{q}\left(s-\left|\alpha _{i} \right|\right) +c_{3} .
\end{equation}
Using \eqref{GrindEQ__3_34_}-\eqref{GrindEQ__3_36_} and the argument similar to that used in Case 2.2 in the proof of Lemma \ref{lem 3.1.}, we can also prove
\begin{equation} \label{GrindEQ__3_37_}
II_{7} \lesssim\left\| |x|^{-b-\left|\alpha '\right|} \left|u\right|^{\sigma -q} \partial _{x_{k} } u\prod _{i=1}^{q}D^{\alpha _{i} } u \right\| _{p_{1} } \lesssim\left\| u\right\| _{\dot{H}_{r}^{s} }^{\sigma +1} ,
\end{equation}
whose proof will be omitted.

Finally, we estimate $II_{8} $. We can see that
\begin{equation} \nonumber
II_{8} \lesssim\left\| |x|^{-b-\left|\alpha '\right|} \left|u\right|^{\sigma -q} \partial _{x_{k} } \left(\prod _{i=1}^{q}D^{\alpha _{i} } u \right)\right\| _{p_{1} } \lesssim\sum _{i=1}^{q}\left\| |x|^{-b-\left|\alpha '\right|} \left|u\right|^{\sigma -q} \partial _{x_{k} } \left(D^{\alpha _{i} } u\right)\prod _{j\in I_{i} }D^{\alpha _{j} } u \right\| _{p_{1} }  ,
\end{equation}
where $I_{i} =\left\{j\in \N:\;1\le j\le q,\;i\ne i\right\}$. Using the fact $s\ge \left|\alpha _{i} \right|+1$ and \eqref{GrindEQ__3_34_}-\eqref{GrindEQ__3_36_}, we can also prove
\begin{equation} \label{GrindEQ__3_38_}
\left\| |x|^{-b-\left|\alpha '\right|} \left|u\right|^{\sigma -q} \partial _{x_{k} } \left(D^{\alpha _{i} } u\right)\prod _{j\in I_{i} }D^{\alpha _{j} } u \right\| _{p_{1} } \lesssim\left\| u\right\| _{\dot{H}_{r}^{s} }^{\sigma +1} .
\end{equation}
Using \eqref{GrindEQ__3_25_}, \eqref{GrindEQ__3_26_}, \eqref{GrindEQ__3_33_}, \eqref{GrindEQ__3_37_}, \eqref{GrindEQ__3_38_}, we can get \eqref{GrindEQ__3_14_}. This completes the proof.
\end{proof}
\begin{rem}\label{rem 3.3.}
\textnormal{If $f\left(z\right)$ is a polynomial in $z$ and $\bar{z}$ satisfying $1<\deg \left(f\right)=1+\sigma $, we can see that the assumption $\left\lceil s\right\rceil \le \sigma +1$ in Lemma \ref{lem 3.1.} and \ref{lem 3.2.} can be removed.}
\end{rem}
\begin{lem}\label{lem 3.4.}
Let $1<p,\;r<\infty $, $s\ge 0$ and $\frac{b}{s} <\sigma $. Suppose that \eqref{GrindEQ__3_2_} is satisfied. Then we have
\begin{equation} \nonumber
\left\| |x|^{-b} \left|u\right|^{\sigma } v\right\| _{p} \lesssim\left\| u\right\| _{\dot{H}_{r}^{s} }^{\sigma } \left\| v\right\| _{r} .
\end{equation}
\end{lem}
\begin{proof}
Putting $\frac{1}{a_{2} }:=\frac{1}{r} -\frac{1}{n} \left(s-\frac{b}{\sigma } \right)$, we have the embedding $\dot{H}_{r}^{s} \subset \dot{H}_{a_{2} }^{\frac{b}{\sigma }}$. Hence it follows from H\"{o}lder inequality and Corollary \ref{cor 2.5.} that
\begin{equation}\nonumber
\left\| |x|^{-b} \left|u\right|^{\sigma } v\right\| _{p} \lesssim\left\| |x|^{-\frac{b}{\sigma } } u\right\| _{a_{2} }^{\sigma } \left\| v\right\| _{r} \lesssim\left\| u\right\| _{\dot{H}_{a_{2} }^{\frac{b}{\sigma}}}^{\sigma } \left\| v\right\|_{r} \lesssim\left\| u\right\| _{\dot{H}_{r}^{s} }^{\sigma } \left\| v\right\| _{r},
\end{equation}
this completes the proof.
\end{proof}
\section{Proofs of main results}
First, we prove Theorem \ref{thm 1.3.}. Using the nonlinear estimates established in Section 2, we have the following lemma.
\begin{lem}\label{lem 4.1.}
Under the assumption of Theorem \ref{thm 1.3.}, we have
\begin{equation} \label{GrindEQ__4_1_}
\left\| |x|^{-b} f(u)\right\| _{L^{\gamma \left(p\right)^{{'} } } \left(I,\;\dot{H}_{p'}^{s} \right)} \lesssim\left\| u\right\| _{L^{\gamma \left(r\right)} \left(I,\;\dot{H}_{r}^{s} \right)}^{\sigma +1} ,
\end{equation}
\begin{equation} \label{GrindEQ__4_2_}
\left\| |x|^{-b} f(u)\right\| _{L^{\gamma \left(p\right)^{{'} } } \left(I,\;H_{p'}^{s} \right)} \lesssim\left\| u\right\| _{L^{\gamma \left(r\right)} \left(I,\;\dot{H}_{r}^{s} \right)}^{\sigma } \left\| u\right\| _{L^{\gamma \left(r\right)} \left(I,\;H_{r}^{s} \right)} ,
\end{equation}
\begin{eqnarray}\begin{split} \label{GrindEQ__4_3_}
&\left\| |x|^{-b} f(u)-|x|^{-b} f(v)\right\| _{L^{\gamma \left(p\right)^{{'} } } \left(I,\;L^{p'} \right)}\\
&~~~~~~~~~~~~~~~~~\lesssim\left(\left\| u\right\| _{L^{\gamma \left(r\right)} \left(I,\;\dot{H}_{r}^{s} \right)}^{\sigma } +\left\| u\right\| _{L^{\gamma \left(r\right)} \left(I,\;\dot{H}_{r}^{s} \right)}^{\sigma } \right)\left\| u-v\right\| _{L^{\gamma \left(r\right)} \left(I,\;L^{r} \right)} ,
\end{split}\end{eqnarray}
where $I\subset \R$ is an interval, $p=\frac{2n}{n-2} $ and $r$ is given in \eqref{GrindEQ__1_9_}.
\end{lem}
\begin{proof}
Noticing that
\begin{equation} \label{GrindEQ__4_4_}
\frac{1}{2} -\frac{1}{n} <\frac{1}{r} =\frac{1}{2} -\frac{1}{n\left(\sigma +1\right)} <\frac{1}{2} ,
\end{equation}
we can see that $\left(\gamma \left(r\right),\;r\right)$ is an admissible pair. We can also see that $\frac{1}{r} -\frac{s}{n} >0$ if, and only if, $b<1+\frac{n-2s}{2} $. Note also that $\frac{b+1}{s} <\sigma $ is equivalent to $b<\frac{6s}{n} -1$, and that $\frac{b}{s} <\sigma $ is equivalent to $b<\frac{4s}{n} $. We also have
\begin{equation} \label{GrindEQ__4_5_}
\frac{1}{p'} =\sigma \left(\frac{1}{r} -\frac{s}{n} \right)+\frac{1}{r} +\frac{b}{n} .
\end{equation}
Furthermore, we can easily verify that $\frac{1}{p'}+\frac{\left\lceil s\right\rceil-s}{n} <1$ for $s\notin \N$ and $n\ge 4$. Hence it follows from Lemma \ref{lem 3.1.} and \ref{lem 3.2.} that
\begin{equation} \label{GrindEQ__4_6_}
\left\| |x|^{-b} f(u)\right\| _{\dot{H}_{p'}^{s} } \lesssim\left\| u\right\| _{\dot{H}_{r}^{s} }^{\sigma +1} .
\end{equation}
On the other hand, we have
\begin{equation} \label{GrindEQ__4_7_}
\frac{1}{\gamma (p)'} =\frac{\sigma +1}{\gamma (r)} .
\end{equation}
Using \eqref{GrindEQ__4_6_} and \eqref{GrindEQ__4_7_}, we have
\begin{equation} \label{GrindEQ__4_8_}
\left\| |x|^{-b} f(u)\right\| _{L^{\gamma \left(p\right)^{{'} } } \left(I,\;\dot{H}_{p'}^{s} \right)} \lesssim\left\| u\right\| _{L^{\gamma \left(r\right)} \left(I,\;\dot{H}_{r}^{s} \right)}^{\sigma +1} ,
\end{equation}
this concludes the proof of \eqref{GrindEQ__4_1_}. It follows from \eqref{GrindEQ__4_6_} and Lemma \ref{lem 3.4.} that
\begin{equation} \label{GrindEQ__4_9_}
\left\| |x|^{-b} f(u)\right\| _{H_{p'}^{s} } \lesssim\left\| u\right\| _{\dot{H}_{r}^{s} }^{\sigma } \left\| u\right\| _{H_{r}^{s} } .
\end{equation}
Using \eqref{GrindEQ__4_7_}, \eqref{GrindEQ__4_9_} and H\"{o}lder inequality, we also have
\begin{equation} \label{GrindEQ__4_10_}
\left\| |x|^{-b} f(u)\right\| _{L^{\gamma \left(p\right)^{{'} } } \left(I,\;H_{p'}^{s} \right)} \lesssim\left\| u\right\| _{L^{\gamma \left(r\right)} \left(I,\;\dot{H}_{r}^{s} \right)}^{\sigma } \left\| u\right\| _{L^{\gamma \left(r\right)} \left(I,\;H_{r}^{s} \right)} ,
\end{equation}
this conclude the proof of \eqref{GrindEQ__4_2_}. Finally, we prove \eqref{GrindEQ__4_3_}. Using the same argument as in Remark 2.6 in \cite{G17}, we can easily see that
\begin{equation} \label{GrindEQ__4_11_}
\left||x|^{-b} f(u)-|x|^{-b} f(v)\right|\lesssim|x|^{-b} \left(\left|u\right|^{\sigma } +\left|v\right|^{\sigma } \right)\left|u-v\right|.
\end{equation}
It follows from \eqref{GrindEQ__4_11_} and Lemma \ref{lem 3.4.} that
\begin{equation} \label{GrindEQ__4_12_}
\left\| |x|^{-b} f(u)-|x|^{-b} f(v)\right\| _{L^{p'} } \lesssim\left(\left\| u\right\| _{\dot{H}_{r}^{s} }^{\sigma } +\left\| v\right\| _{\dot{H}_{r}^{s} }^{\sigma } \right)\left\| u-v\right\| _{L^{r} } .
\end{equation}
Using \eqref{GrindEQ__4_7_}, \eqref{GrindEQ__4_12_} and H\"{o}lder inequality, we immediately have
\begin{eqnarray}\begin{split} \label{GrindEQ__4_13_}
&\left\| |x|^{-b}f(u)-|x|^{-b}f(v)\right\| _{L^{\gamma \left(p\right)^{{'} } } \left(I,\;L^{p'} \right)} \\
&~~~~~~~~~~~~~~~~~~\lesssim\left(\left\| u\right\| _{L^{\gamma \left(r\right)} \left(I,\;\dot{H}_{r}^{s} \right)}^{\sigma } +\left\| u\right\| _{L^{\gamma \left(r\right)} \left(I,\;\dot{H}_{r}^{s} \right)}^{\sigma } \right)\left\| u-v\right\| _{L^{\gamma \left(r\right)} \left(I,\;L^{r} \right)} ,
\end{split}\end{eqnarray}
this completes the proof.
\end{proof}

\begin{proof}[\textbf{Proof of Theorem \ref{thm 1.3.}}]

First, we prove the local well-posedness of \eqref{GrindEQ__1_1_}. Let $T>0$ and $A>0$ which will be chosen later. We define the following complete metric space
\begin{equation} \label{GrindEQ__4_14_}
D=\left\{u\in L^{\gamma \left(r\right)} \left(I,\;H_{r}^{s} \right):\;\left\| u\right\| _{L^{\gamma \left(r\right)} \left(I,\;H_{r}^{s} \right)} \le A\right\}
\end{equation}
which is equipped with the metric
\begin{equation} \label{GrindEQ__4_15_}
d\left(u,\;v\right)=\left\| u-v\right\| _{L^{\gamma \left(r\right)} \left(I,\;L^{r} \right)} ,
\end{equation}
where $I=\left[-T,\;T\right]$ and $r$ is given in \eqref{GrindEQ__1_9_}. We consider the mapping
\begin{equation} \label{GrindEQ__4_16_}
T:\;u(t)\to S(t)u_{0} -i\lambda \int _{0}^{t}S(t-\tau )|x|^{-b} \left|u(\tau )\right|^{\sigma } u(\tau )d\tau  \equiv u_{L} +u_{NL} ,
\end{equation}
where
\[u_{L} =S(t)u_{0} ,~u_{NL} =-i\lambda \int _{0}^{t}S(t-\tau )|x|^{-b} \left|u(\tau )\right|^{\sigma } u(\tau )d\tau  .\]
By Strichartz estimate \eqref{GrindEQ__2_1_}, we can see that $\left\| S(t)u_{0} \right\| _{L^{\gamma (r)} ([-T,\;T],\;H_{r}^{s} )} \to 0$ as $T\to 0$.
We take $A>0$ satisfying $CA^{\sigma } \le \frac{1}{4} $ and $T>0$ such that
\begin{equation} \label{GrindEQ__4_17_}
\left\| S(t)u_{0} \right\| _{L^{\gamma (r)} ([-T,\;T],\;H_{r}^{s} )} \le \frac{A}{2} .
\end{equation}
Using Strichartz estimates \eqref{GrindEQ__2_2_} and nonlinear estimate \eqref{GrindEQ__4_2_}, we have
\begin{equation} \label{GrindEQ__4_18_}
\left\| u_{NL} \right\| _{L^{\gamma \left(r\right)} \left(I,\;H_{r}^{s} \right)} \lesssim\left\| u\right\| _{L^{\gamma \left(r\right)} \left(I,\;\dot{H}_{r}^{s} \right)}^{\sigma } \left\| u\right\| _{L^{\gamma \left(r\right)} \left(I,\;H_{r}^{s} \right)} .
\end{equation}
In view of \eqref{GrindEQ__4_17_} and \eqref{GrindEQ__4_18_}, we have
\begin{equation} \label{GrindEQ__4_19_}
\left\| Tu\right\| _{L^{\gamma \left(r\right)} \left(I,\;H_{r}^{s} \right)} \le \left\| S(t)u_{0} \right\| _{L^{\gamma \left(r\right)} \left(I,\;H_{r}^{s} \right)} +C\left\| u\right\| _{L^{\gamma \left(r\right)} \left(I,\;H_{r}^{s} \right)}^{\sigma +1} \le A.
\end{equation}
It also follow from \eqref{GrindEQ__2_2_} and \eqref{GrindEQ__4_3_} that
\begin{eqnarray}\begin{split} \label{GrindEQ__4_20_}
\left\| Tu-Tv\right\| _{L^{\gamma \left(r\right)} \left(I,\;L^{r} \right)}& \lesssim\left(\left\| u\right\| _{L^{\gamma \left(r\right)} \left(I,\;\dot{H}_{r}^{s} \right)}^{\sigma } +\left\| u\right\| _{L^{\gamma \left(r\right)} \left(I,\;\dot{H}_{r}^{s} \right)}^{\sigma } \right)\left\| u-v\right\| _{L^{\gamma \left(r\right)} \left(I,\;L^{r} \right)}  \\
&\le 2CA^{\sigma } \left\| u-v\right\| _{L^{\gamma \left(r\right)} \left(I,\;L^{r} \right)} \le \frac{1}{2} \left\| u-v\right\| _{L^{\gamma \left(r\right)} \left(I,\;L^{r} \right)} .
\end{split}\end{eqnarray}
\eqref{GrindEQ__4_19_} and \eqref{GrindEQ__4_20_} imply that $T:(D,\;d)\to (D,\;d)$ is a contraction mapping. From Banach fixed point theorem, there exists a unique solution $u$ of \eqref{GrindEQ__1_1_} in $(D,d)$. Furthermore, for any admissible pair $(\gamma(p),\;p)$, it follows from Lemma \ref{lem 2.6.} (Strichartz estimates) and Lemma \ref{lem 4.1.} that
\begin{equation} \nonumber
\left\| u\right\| _{L^{\gamma \left(p\right)} \left(I,\;H_{p}^{s} \right)} \lesssim\left\| u_{0} \right\| _{H^{s} } +\left\| u\right\| _{L^{\gamma \left(r\right)} \left(I,\;H_{r}^{s} \right)}^{\sigma +1},
\end{equation}
which implies $u\in L^{\gamma \left(p\right)} \left(I,\;H_{p}^{s} \right)$. This completes the proof of the local well-posedness of \eqref{GrindEQ__1_1_}.

Next, we consider the global well-posedness of \eqref{GrindEQ__1_1_} with small initial data. Let $M>0$ and $m>0$ which will be chosen later. We define the following complete metric space
\begin{equation} \label{GrindEQ__4_21_}
\bar{D}=\left\{u\in L^{\gamma \left(r\right)} \left(\R,\;H_{r}^{s} \right):\;\left\| u\right\| _{L^{\gamma \left(r\right)} \left(\R,\;\dot{H}_{r}^{s} \right)} \le m,\;\left\| u\right\| _{L^{\gamma \left(r\right)} \left(\R,\;H_{r}^{s} \right)} \le M\right\}
\end{equation}
which is equipped with the metric
\begin{equation} \label{GrindEQ__4_22_}
d\left(u,\;v\right)=\left\| u-v\right\| _{L^{\gamma \left(r\right)} \left(\R,\;L^{r} \right)} .
\end{equation}
It follows from \eqref{GrindEQ__4_16_}, Lemma \ref{lem 2.6.} (Strichartz estimates) and Lemma \ref{lem 4.1.} that
\begin{equation} \label{GrindEQ__4_23_}
\left\| Tu\right\| _{L^{\gamma \left(r\right)} \left(\R,\;\dot{H}_{r}^{s} \right)} \le C\left(\left\| u_{0} \right\| _{\dot{H}^{s} } +\left\| u\right\| _{L^{\gamma \left(r\right)} \left(\R,\;\dot{H}_{r}^{s} \right)}^{\sigma +1} \right),
\end{equation}
\begin{equation} \label{GrindEQ__4_24_}
\left\| Tu\right\| _{L^{\gamma \left(r\right)} \left(\R,\;H_{r}^{s} \right)} \le C\left(\left\| u_{0} \right\| _{H^{s} } +\left\| u\right\| _{L^{\gamma \left(r\right)} \left(\R,\;\dot{H}_{r}^{s} \right)}^{\sigma } \left\| u\right\| _{L^{\gamma \left(r\right)} \left(\R,\;H_{r}^{s} \right)} \right),
\end{equation}
\begin{equation} \label{GrindEQ__4_25_}
\left\| Tu-Tv\right\| _{L^{\gamma \left(r\right)} \left(\R,\;L^{r} \right)} \le C\left(\left\| u\right\| _{L^{\gamma \left(r\right)} \left(\R,\;\dot{H}_{r}^{s} \right)}^{\sigma } +\left\| u\right\| _{L^{\gamma \left(r\right)} \left(\R,\;\dot{H}_{r}^{s} \right)}^{\sigma } \right)\left\| u-v\right\| _{L^{\gamma \left(r\right)} \left(\R,\;L^{r} \right)} .
\end{equation}
Put $m=2C\left\| u_{0} \right\| _{\dot{H}^{s} } $, $M=2C\left\| u_{0} \right\| _{H^{s} } $ and $\delta =2\left(4C\right)^{-\frac{\sigma +1}{\sigma } } $. If $\left\| u_{0} \right\| _{\dot{H}^{s} } \le \delta $, i.e. $Cm^{\sigma } <\frac{1}{4} $, then it follows from \eqref{GrindEQ__4_23_}-\eqref{GrindEQ__4_25_} that
\begin{equation} \label{GrindEQ__4_26_}
\left\| Tu\right\| _{L^{\gamma \left(r\right)} \left(\R,\;\dot{H}_{r}^{s} \right)} \le \frac{m}{2} +Cm^{\sigma +1} \le m,
\end{equation}
\begin{equation} \label{GrindEQ__4_27_}
\left\| Tu\right\| _{L^{\gamma \left(r\right)} \left(\R,\;H_{r}^{s} \right)} \le \frac{m}{2} +Cm^{\sigma } M\le M,
\end{equation}
\begin{equation} \label{GrindEQ__4_28_}
\left\| Tu-Tv\right\| _{L^{\gamma \left(r\right)} \left(\R,\;L^{r} \right)} \le 2Cm^{\sigma } \left\| u-v\right\| _{L^{\gamma \left(r\right)} \left(\R,\;L^{r} \right)} \le \frac{1}{2} \left\| u-v\right\| _{L^{\gamma \left(r\right)} \left(\R,\;L^{r} \right)} .
\end{equation}
Hence, $T:(\bar{D},\;d)\to (\bar{D},\;d)$ is a contraction mapping and there exists a unique solution of \eqref{GrindEQ__1_1_} in $\bar{D}$.
Furthermore for any admissible pair $(\gamma (p),\;p)$, it follows from Lemma \ref{lem 2.6.} (Strichartz estimates) and Lemma \ref{lem 4.1.} that
\begin{equation} \label{GrindEQ__4_29_}
\left\| u\right\| _{L^{\gamma (p)} \left(\R,\;\dot{H}_{p}^{s} \right)} \lesssim\left\| u_{0} \right\| _{\dot{H}^{s} } +\left\| u\right\| _{L^{\gamma \left(r\right)} \left(\R,\;\dot{H}_{r}^{s} \right)}^{\sigma +1} \le m=2C\left\| u_{0} \right\| _{\dot{H}^{s} } ,
\end{equation}
\begin{equation} \label{GrindEQ__4_30_}
\left\| u\right\| _{L^{\gamma (p)} \left(\R,\;H_{p}^{s} \right)} \lesssim\left\| u_{0} \right\| _{H^{s} } +\left\| u\right\| _{L^{\gamma \left(r\right)} \left(\R,\;\dot{H}_{r}^{s} \right)}^{\sigma } \left\| u\right\| _{L^{\gamma \left(r\right)} \left(\R,\;H_{r}^{s} \right)} \le M=2C\left\| u_{0} \right\| _{H^{s} } .
\end{equation}

Finally, we prove the scattering result. We can see that \eqref{GrindEQ__1_13_} is equivalent to
\[{\mathop{\lim }\limits_{t\to \pm \infty }} \left\| e^{-it\Delta } u(t)-u_{0}^{\pm } \right\| _{H^{s} (\R^{n})} =0.\]
In other words, it suffices to show that $e^{-it\Delta } u(t)$ converges in $H^{s} $ as $t_{1} ,\;t_{2} \to \pm \infty $.
Let $0<t_{1} <t_{2} <+\infty $. By using Lemma \ref{lem 2.6.} (Strichartz estimates) and Lemma \ref{lem 4.1.}, we have
\begin{eqnarray}\begin{split} \label{GrindEQ__4_31_}
&\left\| e^{-it_{2} \Delta } u\left(t_{2} \right)-e^{-it_{1} \Delta } u\left(t_{1} \right)\right\| _{H^{s} }=\left\| \int _{t_{1} }^{t_{2} }e^{-i\tau \Delta } |x|^{-b} f\left(u\left(\tau \right)\right)d\tau  \right\| _{H^{s} } \\
&~~~~~~~~~~~~~~~~~~~\lesssim\left\| |x|^{-b} f(u)\right\| _{L^{\gamma \left(p\right)^{{'} } } \left(\left(t_{1} ,\; t_{2} \right),\;H_{p'}^{s} \right)} \lesssim\left\| u\right\| _{L^{\gamma \left(r\right)} \left(\left(t_{1} ,\; t_{2} \right),\;H_{r}^{s} \right)}^{\sigma +{\rm 1}}.
\end{split}\end{eqnarray}
Using \eqref{GrindEQ__4_31_} and the fact $\left\| u\right\| _{L^{\gamma \left(r\right)} \left(\R,\;H_{r}^{s} \right)} <\infty $, we have
\[\left\| e^{-it_{2} \Delta } u\left(t_{2} \right)-e^{-it_{1} \Delta } u\left(t_{1} \right)\right\| _{H^{s} (\R^{n})} \to 0,\]
as $t_{1} ,\;t_{2} \to +\infty $. Thus, the limit $u_{0}^{+} :={\mathop{\lim }\limits_{t\to +\infty }} e^{-it\Delta } u\left(t\right)$
exits in $H^{s} (\R^{n})$. This shows the small data scattering for positive time, the one for negative time is treated similarly. This completes the proof.
\end{proof}
Next, we prove Theorem \ref{thm 1.4.}. Since the proof is very similar to that of Theorem \ref{thm 1.3.}, we only sketch the proof.
\begin{proof}[\textbf{Proof of Theorem \ref{thm 1.4.}}]
Using the hypotheses of Theorem \ref{thm 1.4.}, we can easily verify that $2<r<\frac{2n}{n-2} $, where $r$ is given in \eqref{GrindEQ__1_15_}. Since $b<1$ and $0<\varepsilon <1-b$, we can also see that $\frac{1}{r} >\frac{s}{n} $. Putting $p:=\frac{3}{2-s+\varepsilon } $, we also have $2<p<\frac{2n}{n-2} $. We can also see that \eqref{GrindEQ__4_5_} and \eqref{GrindEQ__4_7_} hold. Furthermore, we can see that $\frac{1}{p'}+\frac{\left\lceil s\right\rceil-s}{n} =1-\frac{\varepsilon }{n} <1$.
Thus we can use Lemma \ref{lem 3.2.} to get \eqref{GrindEQ__4_1_}-\eqref{GrindEQ__4_3_} and we omit the details. Repeating the same argument as in the proof of Theorem \ref{thm 1.3.}, we can get the desired results whose proof will be omitted.
\end{proof}

%%%%%%%%%%%%%%%%%%%%%%%%%%%%%%%%%%%%%%%%%%%%%%%%%%%%%%%%%%%%%%%%%%%%%%%%%%%%%

\end{document}